\documentclass[12pt]{amsart}

\usepackage{amssymb}
\usepackage{amsmath}
\usepackage{amsthm}
\usepackage{amsbsy}
\usepackage{bm}
\usepackage[english]{babel}
\usepackage{graphicx}
\usepackage{color}
\usepackage{graphicx}
\usepackage{color}
\usepackage{hyperref}
\usepackage[toc,page]{appendix}

\theoremstyle{plain}
\newtheorem{theorem}{Theorem}[section]
\newtheorem*{theorem*}{Theorem}
\newtheorem{lemma}[theorem]{Lemma}

\theoremstyle{definition}

\theoremstyle{definition}

\newtheorem{remark}{Remark}

\newcommand{\A}{\alpha}
\newcommand{\B}{\beta}

\newcommand{\N}{\mathbb{N}}
\newcommand{\R}{\mathbb{R}}
\newcommand{\g}{\gamma}
\newcommand{\G}{\Gamma}
\newcommand{\e}{\epsilon}

\renewcommand{\H}{{\mathbb{H}}}
\renewcommand{\t}{\tau}

\newcommand{\la}{\langle}
\newcommand{\ra}{\rangle}

\theoremstyle{plain} 
\newcommand{\thistheoremname}{}
\newtheorem*{genericthm*}{\thistheoremname}
\newenvironment{style}[1]
  {\renewcommand{\thistheoremname}{#1}%
   \begin{genericthm*}}
  {\end{genericthm*}}

\begin{document}
\title[Goldman bracket \& length equivalent filling curves]{Goldman bracket \& length equivalent filling curves}

\author{Arpan Kabiraj}

\address{Department of Mathematics\\ 
		Chennai Mathematical Institute\\
		Chennai 603103, India}

\email{akabiraj@cmi.ac.in}

\thanks{This research is supported by NET-CSIR (India) Senior Research Fellowship.}

\begin{abstract}
A pair of  distinct free homotopy classes of closed curves in an orientable surface $F$ with negative Euler characteristic is said to be length equivalent if for any hyperbolic structure on $F$, the length of the geodesic representative of one  class is equal to the length of the geodesic representative of the other class. Suppose $\A$ and $\B$ are two intersecting oriented closed curves on $F$ and $P$ and $Q$ are any two intersection points between them. If the two terms $\la\A*_P\B\ra$ and $\la\A*_Q\B\ra$ in $[\la\A\ra,\la\B\ra]$, the Goldman bracket between them, are the same, then we construct infinitely many pairs of length equivalent curves in $F.$ These pairs correspond to the terms of the Goldman bracket between a power of $\A$ and $\B$. As a special case, our construction shows that given a self-intersecting geodesic $\A$ of $F$ and any self-intersection point $P$ of $\A$, we get a sequence of such pairs. Furthermore if $\A$ is a filling curve then these pairs are also filling.
%Given an orientable hyperbolic surface $F$ with negative Euler characteristic,     any self intersecting geodesic $\A$ on $F$ and any self-intersection point $P$ of $\A$, we construct infinitely many pairs of length equivalent curves in $F$. We also relate these examples with the terms of Goldman bracket.          
\end{abstract}
\maketitle
\section{Introduction }
Let $F$ be an orientable surface with negative Euler characteristic. By \cite[Theorem 1.2]{Pr}, we can endow $F$ with a hyperbolic metric. By a \emph{hyperbolic surface} we mean an orientable surface of finite type with negative Euler characteristic and with a given hyperbolic metric.

Given a hyperbolic surface $F$, there is a bijective correspondence between the set of non-trivial free homotopy classes of oriented closed curves and the set of non-trivial conjugacy classes of $\pi_1(F)$ \cite[Theorem 2.6]{GC}. Also in every non-trivial, non-peripheral free homotopy class of a closed curve, there exists a unique geodesic \cite[Proposition 1.3]{Pr}. Henceforth we use free homotopy classes of oriented closed curves and conjugacy classes of elements in $\pi_1(F)$ interchangeably. Given any oriented closed curve $\B$, we denote both its free homotopy class in $F$ and its conjugacy class in $\pi_1(F)$ by $\la\B\ra$.

Two distinct conjugacy classes $\la\A\ra$  and $\la\B\ra$ in $\pi_1(F)$ are said to be \emph{length equivalent}  if for every hyperbolic metric in $F$, the length of the unique geodesic representative of the free homotopy class $\la\A\ra$ is equal to the length of the unique geodesic representative of the free homotopy class $\la\B\ra$ and $\la\B\ra\neq\la\A^{-1}\ra$. 

%Given any hyperbolic surface $F$, any self-intersecting closed geodesic $\A$ on $F$ and any self intersection point $P$ of $\A$, we construct a sequence of pairs of length equivalent classes in $\pi_1(F)$.
A geodesic $\A$ in $F$ is said to be  \emph{filling curve} if the component of the the complement of $\A$ in $F$ is homeomorphic to a disc or punctured disc or an annulus. Notice that being a filling curve is not homotopy invariant. We call a free homotopy class of a curve $\la\A\ra$ filling if the geodesic in the free homotopy class of $\la\A\ra$ is a filling curve.

Given two free homotopy classes of oriented closed curves $\la\alpha\ra$ and $\la\beta\ra$ in $F$, consider two oriented closed curves  $\A$ and $\B$ representing $\la\alpha\ra$ and $\la\beta\ra$ respectively. Performing a small homotopy if necessary, we assume that $\A$ and $\B$ intersect transversally in double points. Goldman \cite{Gol} defined the bracket of $\la\alpha \ra$ and $\la\beta \ra$ as the sum, $$[\la\alpha\ra,\la\beta\ra]=\sum_{P\in\A\cap\B}\e_P\la\A*_P\B\ra$$ where $\A\cap\B$ denotes the set of all intersection points between $\A$ and $\B$, $\e_P$ denotes the sign of the intersection between $\A$ and $\B$ at $P$ and $(\A*_P\B)$ denotes the loop product of $\A$ and $\B$ at $P$.

\begin{style}{Main Theorem}
Let $F$ be a hyperbolic surface and $\A$ and $\B$ be two closed oriented curves intersecting transversally in double points. Also assume that they intersect each other minimally in their free homotopy class. Let $P_1$ and $P_2$ be two intersection points between them such that $\la\A*_{P_1}\B\ra=\la\A*_{P_2}\B\ra$, i.e. the terms corresponding to $P_1$ and $P_2$ in $[\la\A\ra,\la\B\ra]$ are same. Then there exists $g,h\in\pi_1(F)$ and a positive integer integer $N$ such that for all $n>N$, the conjugacy classes of $\A^n\B^{g}$ and $\A^n \B^h$ are length equivalent. 
\end{style}
 
Let us consider the following special case. Let $\A$ be any self- intersecting geodesic and $P$ be any self-intersection point of $\A$. As the Goldman bracket is a Lie bracket, $[\la\A\ra,\la\A\ra]=0.$ The reason is the following. Let $\A_1$ and $\A_2$ be two curves freely homotopic to $\A$ which intersect each other transversally and $\A_1\cup\A_2$ cobounds a narrow annulus. Then corresponding to each self intersection point $P$ of $\A$, we get two intersection points $P_1$ and $P_2$ between $\A_1$ and $\A_2$ and $\la\A_1*_{P_1}\B\ra=\la\A_1*_{P_2}\B\ra$ with opposite signs.

\begin{figure}[h]
  \centering
    \includegraphics[trim = 80mm 40mm 80mm 40mm, clip, width=6cm]{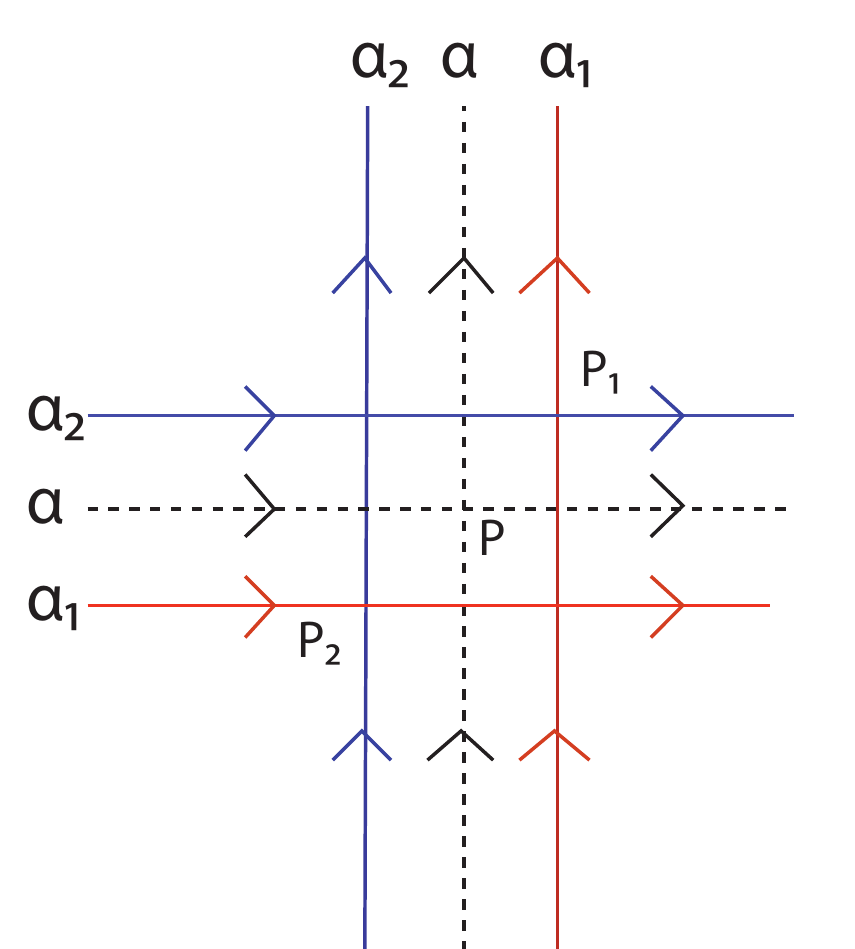}
     \caption{Corresponding to the self-intersection point $P$ of $\A$, there are two intersection points $P_1$ and $P_2$ between $\A_1$ and $\A_2$. The sign of the intersection between $\A_1$ and $\A_2$ at $P_1$ and $P_2$ respectively are opposite to each other.}\label{fattening}
\end{figure}

Now we are in the same situation as in the main theorem and we have the following result.

\begin{theorem}\label{mainlength}Let $F$ be any hyperbolic surface and $\A$ be any self-intersecting closed geodesic in $F$. Given any self intersection point $P$ of $\A$, there exists $g\in \pi_1(F)$ and a positive integer $N$ such that for all $n>N$, the conjugacy classes of $\A^n g\A g^{-1}$ and $g\A^n  g^{-1}\A$ are length equivalent.  Furthermore if $\A$ is a filling curve then $\A^n g\A g^{-1}$ and $g\A^n  g^{-1}\A$ are also filling curves.
\end{theorem}

\subsection*{Length equivalent curves}
Given any hyperbolic surface $F$, the \emph{length spectrum} of $F$ is the set of the lengths of all geodesics in $F$. It is well known that the length spectrum of a hyperbolic surface is discrete \cite[Theorem 1.6.11]{Bu}. Randol \cite{Ra}, using a result of Horowitz \cite{Ho}, proved that on every hyperbolic surface, the length spectrum has  arbitrarily high multiplicities. Buser \cite[Section 3.7]{Bu}, gave an algorithm to construct length equivalent curves in the pairs of pants and in torus with one boundary component. Chas \cite{Ch3} provided examples of length equivalent curves with different self-intersection number. For more details about length equivalent curves, see \cite{A}.

\subsection*{Goldman bracket }
William Goldman \cite{Gol}, discovered the Goldman bracket in the eighties and proved that if the Goldman bracket between two curves is zero and one of them is simple then they have disjoint representatives. Chas \cite{Ch},\cite{Ch1}, proved a stronger version of Goldman's result namely if one of the curves has a simple representative then the number of terms in the Goldman bracket is the same as their geometric intersection number. Chas and Krongold \cite{CK}, proved that for a compact surface with non-empty boundary, the Lie bracket between an element and its cube determines the self-intersection number. Using hyperbolic geometry, Chas and Gadgil \cite{GC}, proved that there exists a positive integer $m$ such that for all $n>m$, the geometric intersection number between $x$ and $y$ is the number of terms in the Goldman bracket between $n$-th power of $x$ and $y$ divided by $n$. For more details about Goldman bracket, see \cite{GC}, \cite{Gol}.

\section*{Acknowledgements} The author would like to thank Siddhartha Gadgil for his encouragement and enlightening conversations. The author would also like to thank Basudeb Datta for his valuable suggestions and comments.      

\section{Notation and definition}

Let $F$ be an orientable surface of negative Euler characteristic, i.e. $F$ is an orientable surface of genus $g$ with $b$ boundary components and $n$ punctures such that $2-2g-b-n<0$. 

Fix a hyperbolic metric on $F$. We identify the fundamental group $\pi_1(F)$ of $F$ with a discrete subgroup of $PSL_2(\R)$, the group of orientation preserving isometries of the upper half plane $\H$. The action of $\pi_1(F)$ on $\H$ is properly discontinuous and without any fixed point and the quotient space is isometric to $F.$ Henceforth by an isometry of $\H$ we mean an orientation preserving isometry. Given $\A$ and $g$ in $\pi_1(F)$, we denote $g\A g^{-1}$ by $\A^g$ and the translation length of $\A$ by $\tau_\A$. If $\A$ is hyperbolic then $\A^g$ is also hyperbolic with $\t_{\A^{g}}=\t_\A$ and $A_{\A^{g}}=gA_\A$ for all $g\in \pi_1(F)$ where $A_\A$ denotes the axis of $\A$.   

By a \emph{lift} of a closed curve $\g$ to $\H$ we mean the image of a lift $\R\rightarrow \H$ of the map $\g\circ\pi$ where $\pi:\R\rightarrow S^1$ is the usual covering map.

Let $\la\A\ra$ be the free homotopy class of an oriented closed self-intersecting  geodesic $\A$. Corresponding to each self intersection point $P$ of $\A$ and each positive integer $n$, we construct two free homotopy classes of oriented closed curves in the following way.   

 Consider a narrow annular neighbourhood of $\A$. Let $\A_1$ (respectively $\A_2$) denote the boundary curve of the annulus at the right hand side (respectively left hand side) of $\A$ (Figure \ref{fattening}). The curves $\A_1$ and $\A_2$ intersect each other transversally and $\A_1\cup\A_2$ cobounds a narrow annulus. Therefore corresponding to each self intersection point $P$ of $\A$, we get two intersection points $P_1$ and $P_2$ between $\A_1$ and $\A_2$.

Let $\G_1(P,n)$ be the loop product of $\A_1^n$ and $\A_2$ at the point $P_1$ and $\G_2(P,n)$ be the loop product of $\A_1^n$ and $\A_2$ at the point $P_2$, where $\g^n$ denotes the curve which goes around the curves $\g$, $n$-times. Observe that the free homotopy class of $\g^n$ is well defined.

Let $\e_i$ denote the sign of the intersection between $\A_1$ and $\A_2$ at $P_i$ for $i=1,2$. From Figure \ref{fattening}, we get $$\e_1=-\e_2.$$  

 In the next section we construct lifts of a curve from the free homotopy class $\la\G_i(P,n)\ra$ to show that the free homotopy class $\la\G_i(P,n)\ra$ depends only on the geodesic $\A$, the self intersection point $P$ and the integer $n$ (Lemma \ref{liftequivalence}). 
 
We prove the following lemma which is the main lemma of this paper. Theorem \ref{mainlength} follows at once from this lemma.
 
\begin{lemma}\label{proxymainlemma}
There exists a positive integer $N$ such that for all $n>N$, the free homotopy classes $\la \G_1(P,n)\ra $ and $\la \G_2(P,n)\ra $ are length equivalent. Furthermore if $\A$ is a filling curve then $\la \G_1(P,n)\ra $ and $\la \G_2(P,n)\ra $ are also filling curves.
\end{lemma}   

\begin{remark} There is a geometric way to see that the geodesics in the free homotopy classes of $\la \G_1(P,n)\ra $ and $\la \G_2(P,n)\ra $ are of equal length. 

First observe that any non-simple closed geodesic $\A$ is the immersion of a figure eight geodesic. Furthermore, you can choose any self-intersection
point of $\A$ to be the image of the self-intersection point of the figure eight geodesic. This immersion extends to an immersed pair of pants (for instance by taking a $\e$-neighbourhood of the immersed curve). Denote by $Y$ the lift of the immersed pants by the immersion.
Now it is easy to see that the lifts of $\la \G_1(P,n)\ra $ and $\la \G_2(P,n)\ra $ to the pair of pants $Y$ are reflected images of one another on $Y$ and thus have the same length. Showing they are not conjugate to each other in $\pi_1(F)$ is not straightforward. 
\end{remark}

 \section{Lifts of the curves}

In this section we construct a lift of a piecewise geodesic curve in the free homotopy class $\la\G_i(P,n)\ra$ and describe the unique geodesic in the free homotopy class $\la\G_i(P,n)\ra$. 

\begin{figure}[h]
  \centering
    \includegraphics[trim = 40mm 45mm 40mm 20mm, clip, width=8cm]{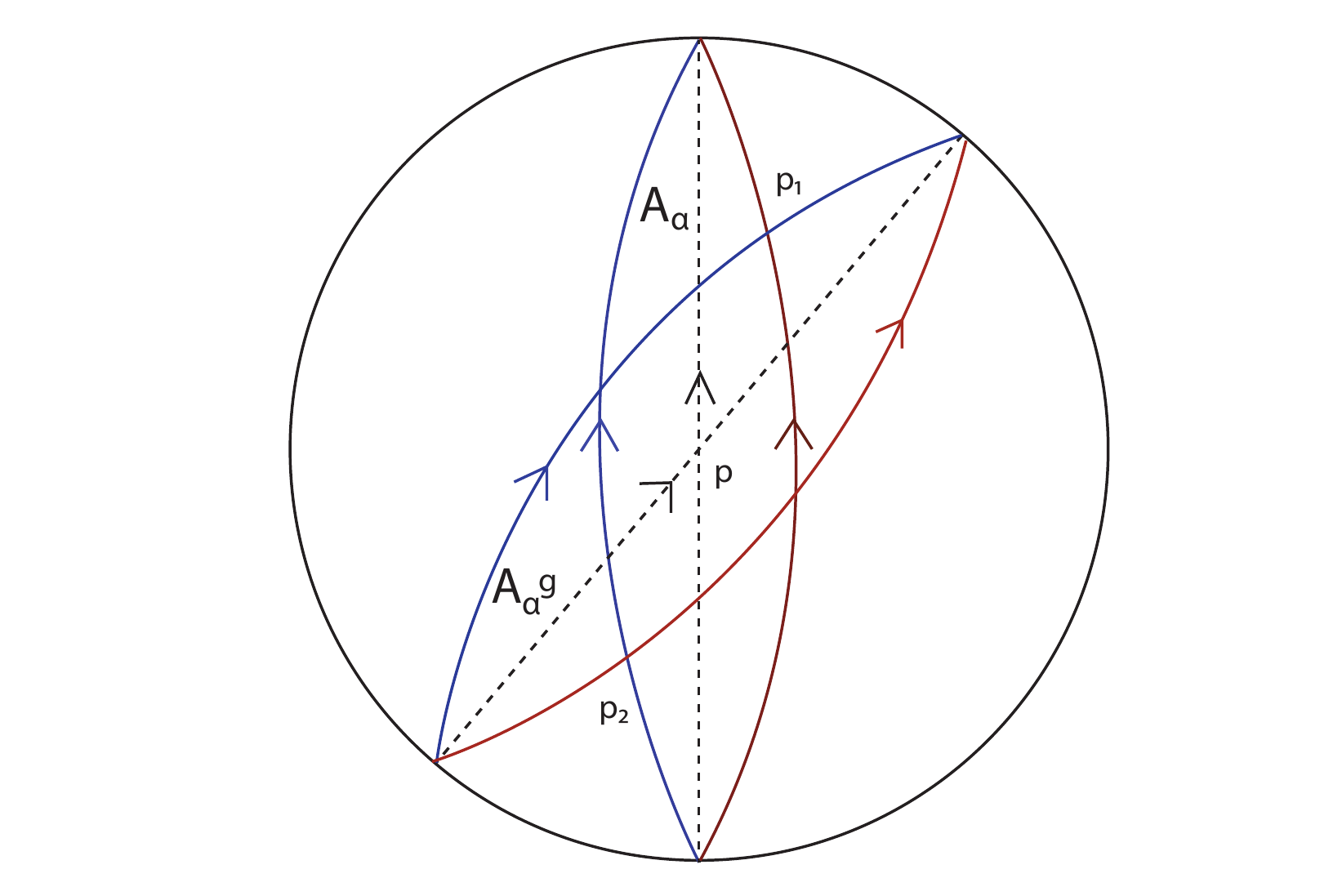}
     \caption{Lift of the Figure \ref{fattening} to the unit disc.}\label{liftoffattening}
\end{figure}
Let $\A$ be a self-intersecting geodesic in an oriented hyperbolic surface $F$. Let $P$ be a self-intersection point of $\A$. We consider $P$ to be the base-point of the fundamental group of $F$. Therefore we consider $\A$ to be an element of $\pi_1(F)$.

Let $A_{\A}$ be the axis of $\A$ in the upper half plane $\H$. Let $p$ be a lift of the self-intersection point $P$ on $A_{\A}$. By \cite[Section-3]{GC}, there exists $g$ in $\pi_1(F)$ such that $gA_{\A}\cap A_{\A}=\{p\}$. The axis $gA_{\A}$ is same as the axis of ${\A^g}$. 

From Figure \ref{liftoffattening}, the sign of the intersection between $A_{\A}$ and $A_{\A^g}$ at $p$ is $\e_1$ and the sign of the intersection between $A_{\A^g}$ and $A_{\A}$ at $p$ is $\e_2$. 

Let $\B_1:[0,1]\rightarrow \H$ be a curve  from $\B_1(0)=(\A^g)^{-1}p$ to $\B_1(1)=\A^np$, whose image is given by the concatenation of the geodesic segment of $A_{\A^g}$ from $(\A^g)^{-1}p$ to $p$ with the geodesic segment of $A_{\A}$ from $p$ to $\A^np$. As $\A^n\A^g(\B_1(0))=\B_1(1)$, we  extend $\B_1$ periodically  to a map $\g_1(p,n):\R\rightarrow\H$ such that $\g_1(p,n)(t)=\B_1(t)$ for $t\in[0,1]$ and $\g_1(p,n)(t+1)=\A^n\A^g(\g_1(p,n)(t))$ for all $t$.  

Similarly let $\B_2:[0,1]\rightarrow \H$ be a curve  from $\B_2(0)=(\A)^{-1}p$ to $\B_2(1)=(\A^g)^np$, whose image is given by the concatenation of the geodesic segment of $A_{\A}$ from $(\A)^{-1}p$ to $p$ with the geodesic segment of $A_{\A^g}$ from $p$ to $(\A^g)^np$. As $(\A^g)^n\A(\B_2(0))=\B_2(1)$, we  extend $\B_2$ periodically  to a map $\g_2(p,n):\R\rightarrow\H$ such that $\g_2(p,n)(t)=\B_2(t)$ for $t\in[0,1]$ and $\g_2(p,n)(t+1)=(\A^g)^n\A(\g_2(p,n)(t))$ for all $t$.

We call the lifts of $P$ in $\g_i(p,n)$ as the \emph{vertices} of $\g_i(p,n)$. 

\begin{remark}\label{proj}
As $\A$ is freely homotopic to $\A_1$ and $\A_2$, the projection of $\B_1$ is freely homotopic to $\G_1(P,n)$ and the projection of $\B_2$ is freely homotopic to $\G_2(P,n)$. 
\end{remark}

\begin{remark}\label{intersection}
Let $\mu_1=\A^n\A^g$ and $\mu_2=(\A^g)^n\A$. Denote the geodesic arc from $P$ to $\A^n(P)$ by $I_1$, the geodesic segment from $\A^n(P)$ to $\A^n\A^g(P)$ by $I_2$, the geodesic segment from $P$ to $(\A^g)^n$ by $J_1$ and the geodesic segment from $(\A^g)^n(P)$ to $(\A^g)^n\A (P)$ by $J_2$. Then $\g_1(p,n)$ consists of geodesic segments of the form $\mu_1^k(I_1)$ and $\mu_2^k(I_2)$ occurring alternately. Similarly $\g_2(p,n)$ consists of geodesic segments of the form $\mu_1^k(J_1)$ and $\mu_2^k(J_2)$ occurring alternately. 

Every vertex of $\g_1(p,n)$ corresponds to an intersection point between $\mu_1^k(I_1)$ and $\mu_2^k(I_2)$. Similarly, every vertex of  $\g_2(p,n)$ corresponds to an intersection point between  $\mu_1^k(J_1)$ and $\mu_2^k(J_2)$.
\end{remark}

\begin{remark}\label{signremark}  The sign of the intersection  between $\mu_1^k(I_1)$ and $\mu_2^k(I_2)$ at any  vertex of $\g_1(p,n)$ is equal to $\epsilon_1$. Similarly, the sign of the intersection between $\mu_1^k(J_1)$ and $\mu_2^k(J_2)$ at any vertex of $\g_2(p,n)$, is equal to $\epsilon_2$. 
\end{remark} 

 \begin{figure}[h]
  \centering
    \includegraphics[trim = 40mm 10mm 40mm 10mm, clip, width=6cm]{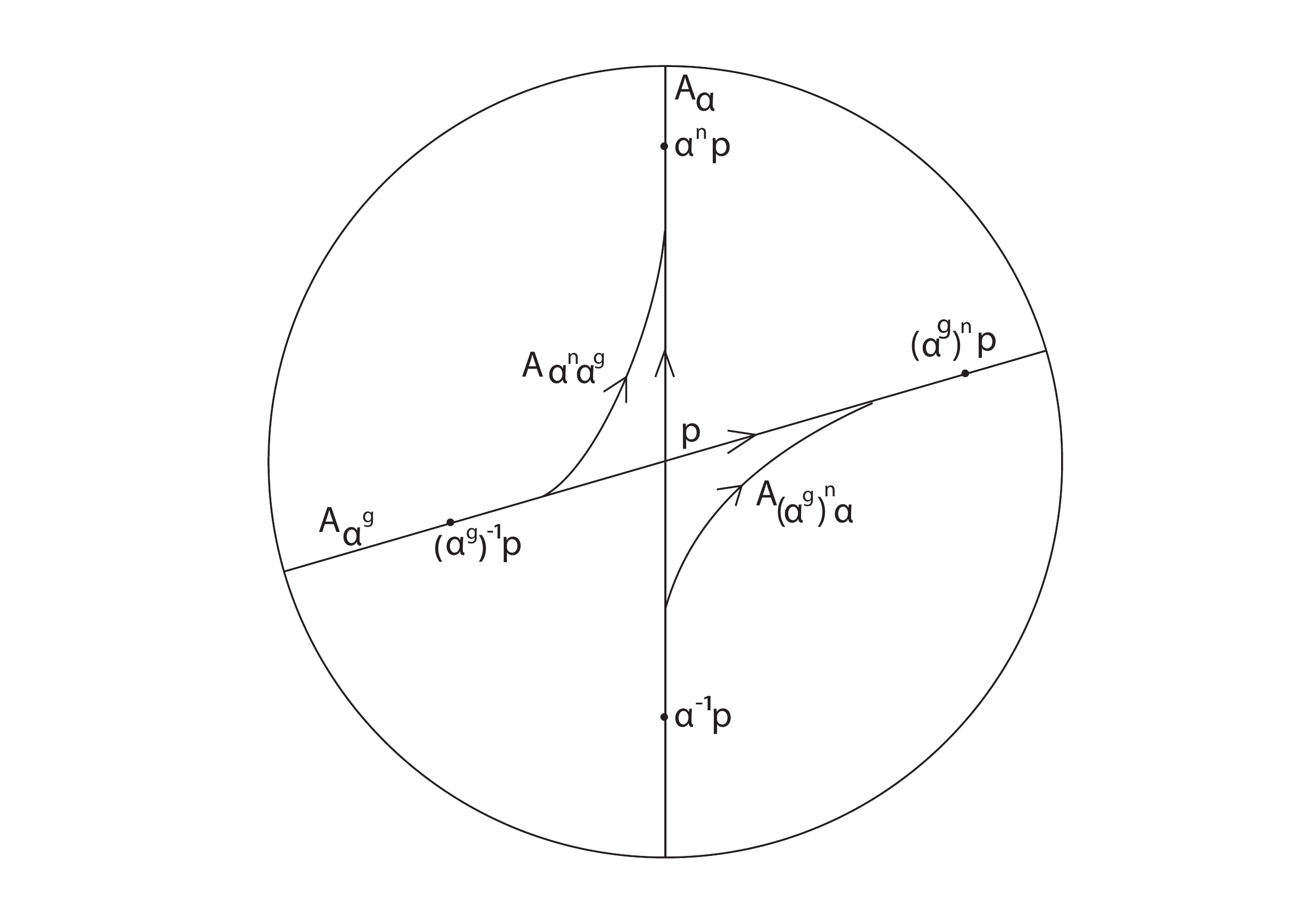}
     \caption{}\label{geodesiclift}
\end{figure} 

 By \cite[Lemma 7.1]{GC}, $\g_i(p,n)$ is a quasi-geodesic for $i\in\{1,2\}$. As every quasi geodesic in $\H$ is homotopic (fixing endpoints) to a unique geodesic, by Remark \ref{proj},  the unique geodesic corresponding to the free homotopy class $\la\G_i(P,n)\ra$ is the projection of the axis of the quasi geodesic $\g_i(p,n)$.
 
 By \cite[Theorem 7.38.6]{B} and Figure \ref{geodesiclift}, the axis of $\g_1(p,n)$ is the axis of $\A^n\A^g$ and the axis of $\g_2(p,n)$ is the axis of $(\A^g)^n\A$. 
 
 Therefore the free homotopy class  $\la\G_1(P,n)\ra$ corresponds to the conjugacy class of $\A^n\A^g$ and the free homotopy class  $\la\G_1(P,n)\ra$ corresponds to the conjugacy class of  $(\A^g)^n\A$.
 
 From the above discussion we obtain the following lemma. 
 
 \begin{lemma}\label{liftequivalence} Given any self-intersecting geodesic $\A$ in $F$, any self-intersection point $P$ of $\A$ and any positive integer $n$, the two free homotopy classes $\la\G_1(P,n)\ra$ and $\la\G_2(P,n)\ra$ correspond to the conjugacy classes of $\A^n\A^g$ and $(\A^g)^n\A$ respectively, where $g\in \pi_1(F)$ is determined by a fixed lift p of $P$ on $A_{\A}$. 
 \end{lemma}

\section{Final results}

\begin{lemma}\label{le1} Let $\A$ and $g$ be two elements in $\pi_1(F)$ such that $A_{\A}$ intersects $A_{\A^g}$. Then for any hyperbolic metric $X$ on $F$,  $l_X(\A^n\A^g)=l_X((\A^g)^n\A)$, where $l_{X}(\A)$ denotes the length of the geodesic freely homotopic to $\A$ in the metric $X$.  
\end{lemma}
\begin{proof}
We give two proofs of this lemma. The first one is geometric and the second one is algebraic.

\textbf{Geometric proof:} 
Let $X$ be any hyperbolic metric in $F$ and $\theta$ be the angle between $A_{\A}$ and $A_{\A^g}$ at the point $P$ in the positive direction of both axes. Consider the triangles $\triangle pqr$ and $\triangle pq_1r_1$ in Figure \ref{axsign} and the angle $\mu=\pi-\theta$. 

By \cite[Theorem 7.38.6]{B}, we have the following:

\begin{itemize}
\item $l(pq)=l_X(\A^n)/2=n\tau_{\A}/2=n\tau_{\A^g}/2=l_X((\A^g)^n)/2=l(pq_1)$.
\item $l(rp)=l_X(\A^g)/2=\tau_{\A^g}/2=\tau_{\A}/2=l_X(\A)/2=l(r_1p)$.
\item $\angle rpq=\mu=\angle r_1pq_1$.
\item $l(rq)=l_X(\A^n\A^g)/2$ and $l(r_1q_1)=l_X((\A^g)^n\A)/2$.
\end{itemize}

\begin{figure}[h]
  \centering
    \includegraphics[trim = 10mm 10mm 10mm 10mm, clip, width=10cm]{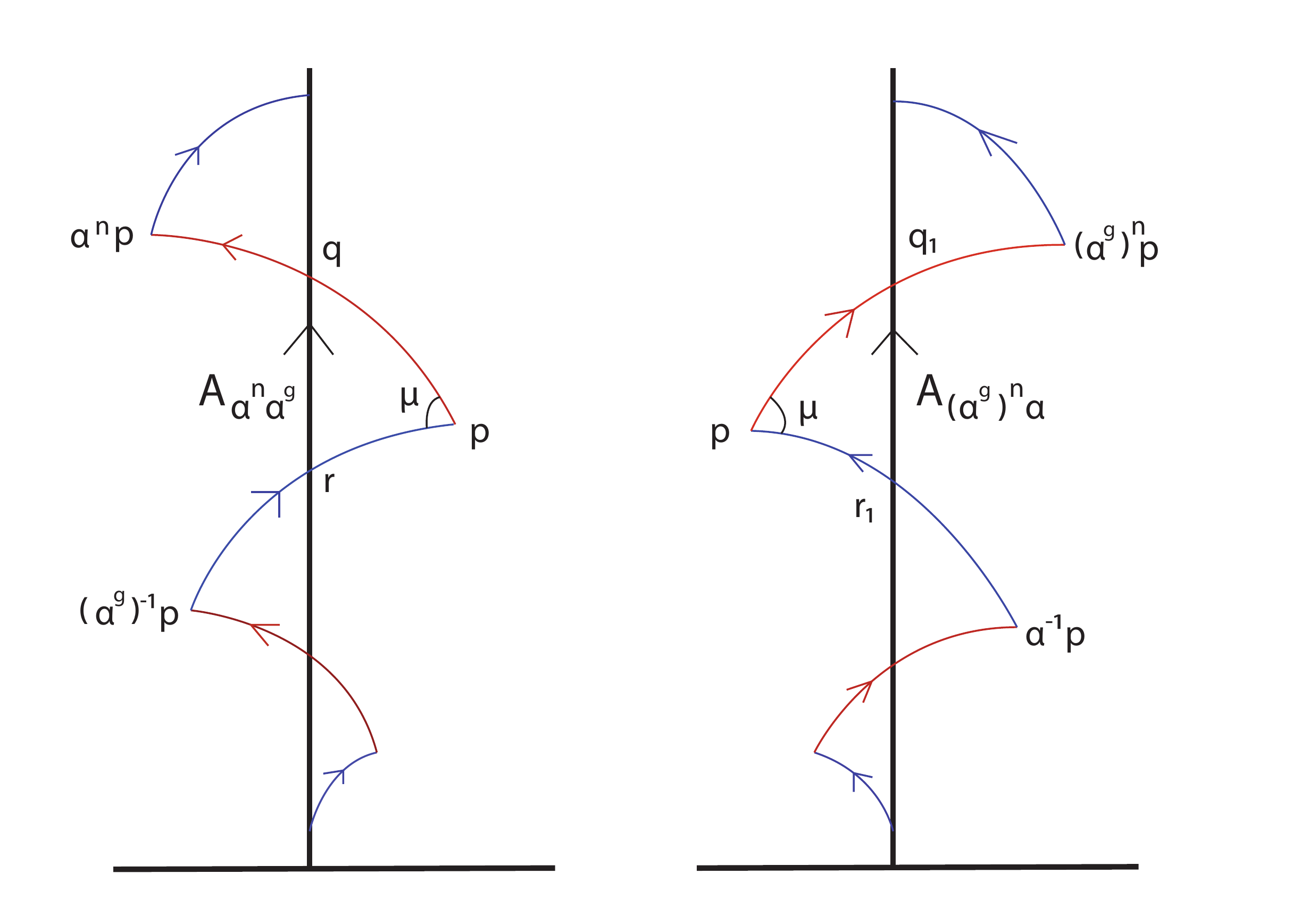}
     \caption{}\label{axsign}
\end{figure}

 Therefore by the hyperbolic cosine rule,  
 $$l_X(\A^n\A^g)=l_X((\A^g)^n\A).$$
\textbf{Algebraic proof:} We use the following equalities of $PSL_2(\R)$ which follows from straightforward computation. For details see \cite[Lemma 3.7.9]{Bu} 

If $A$ and $B$ belong to $PSL_2(\R)$ then the we have the  following equalities.
\begin{align*} 
trace(A) &= trace(A^{-1})   \\ 
trace(AB) &=  trace(BA) \\
trace(A) &= trace(BAB^{-1}) \\
trace(AB) &= trace(A)trace(B)-trace(AB^{-1}) 
\end{align*} 

If $A$ is a matrix representing the geodesic $\A$ in $X$ then $$l_X(\A)=\cosh^{-1}(\frac{|trace(A)|}{2}).$$

Let $A=\A$ and $B=\A^{g}$. By the above comment, it is sufficient to show that $trace(A^{n}B)=trace(B^nA)$ for all $n$ in $\N$. We prove this using induction on $n$. 

For $n=1$, this follows from the second equation. 

Suppose the result is true for all $k<n$. Using the above equations, we have  
\begin{align*} 
trace(A^nB) &= trace(BA^n).   \\ 
 &=  trace(BA^{n-1})trace(A)-trace(BA^{n-2}). \\
&= trace(A^{n-1}B)trace(A)-trace(A^{n-2}B). \\
 &= trace(B^{n-1}A)trace(B)-trace(B^{n-2}A)\\
 &= trace(AB^{n-1})trace(B)-trace(AB^{n-2})\\
 &= trace(AB^n)\\
 &= trace(B^nA). 
\end{align*}
\end{proof}
%\begin{remark} 
%\end{remark}

\begin{lemma}\label{le2}
There exists a positive integer $N>1$ such that for all $n>N$, $\A^n\A^g$ is not conjugate to $(\A^g)^n\A$ or $((\A^g)^n\A)^{-1}$. 
\end{lemma}

\begin{figure}[h]
  \centering
    \includegraphics[trim = 10mm 10mm 10mm 10mm, clip, width=10cm]{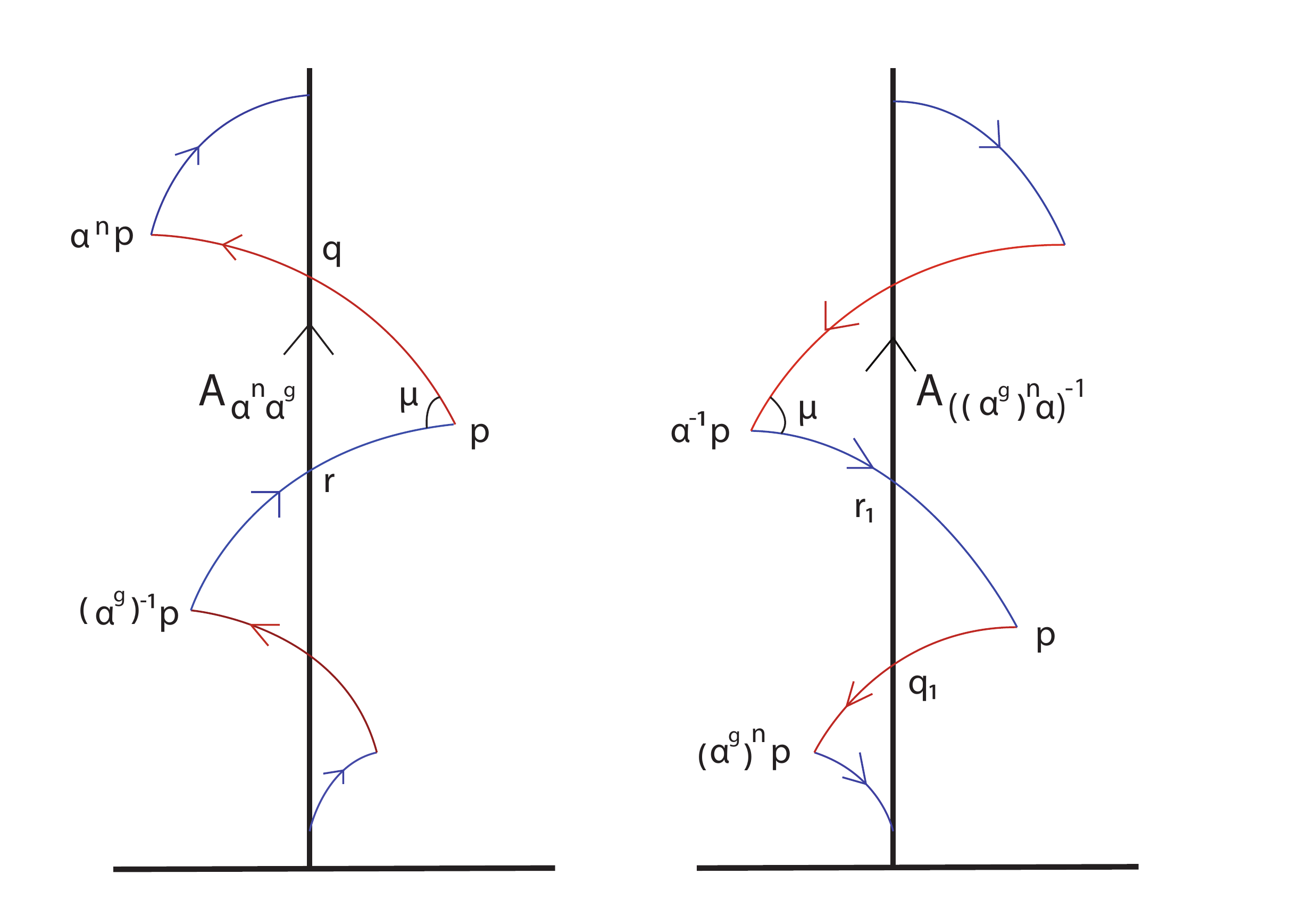}
     \caption{}\label{axsign1}
\end{figure}

\begin{proof}
Conjugacy of two elements in $\pi_1(F)$ is independent of the hyperbolic metric in $F$. So we fix a metric and show that there exists an integer $N>1$ such that for all $n>N$, $\A^n\A^g$ is not conjugate to $(\A^g)^n\A$ or $((\A^g)^n\A)^{-1}$. Then this is true for all hyperbolic metric in $F$.
  
Fix a hyperbolic metric $X$ on $F$. Let $l_{X}(\A)=L$. Choose the positive integer $N$ as in \cite[Lemma 7.4]{GC}, corresponding to the length $L$ and consider $n>N$.

{\bf{\underline{Case-I}}} If $\A^n\A^g$ is conjugate to $(\A^g)^n\A$ then there exists $h$ in $\pi_1(F)$ such that $hA_{\A^n\A^g}=A_{(\A^g)^n\A}$. By \cite[Lemma 7.4]{GC},  $h\g_1(p,n)=\g_2(p,n)$ and hence their vertices coincide. Therefore by Remark \ref{signremark}, $\e_1=\e_2$, a contradiction.

{\bf{\underline{Case-II}}} If $\A^n\A^g$ is conjugate to  $((\A^g)^n\A)^{-1}$ then  there exists $h$ in $\pi_1(F)$ such that   $hA_{\A^n\A^g}=A_{{((\A^g)^n\A})^{-1}}$. Again by \cite[Lemma 7.4]{GC}, $h\g_1(p,n)=\g_2(p,n)$. Therefore from Figure \ref{axsign1} and from the construction of $\g_i(p,n)$, $A_\A$ coincides with some translate of  $A_{\A^{-1}}$. This implies, $\A$ and $\A^{-1}$ are conjugate to each other in $\pi_1(F)$ which contradicts the fact that a hyperbolic element is never conjugate to its inverse in $PSL_2(\R)$. Hence $\A^n\A^g$ is not conjugate to  $((\A^g)^n\A)^{-1}$.
\end{proof}

\begin{lemma}\label{filling}
If $\A$ is a filling curves then there exists a positive integer $N>1$ such that for all $n>N$, $\A^n\A^g$ and $(\A^g)^n\A$ are filling curves. 
\end{lemma}
\begin{proof}
We show that if $z$ is a simple closed geodesic such that $i(z,\A)\neq 0$ then $i(z,(\A^g)^n\A)\neq 0$. The proof for the case $i(z, \A^n\A^g)\neq 0$ is similar.

\begin{figure}[h]
  \centering
    \includegraphics[trim = 40mm 10mm 40mm 10mm, clip, width=8cm]{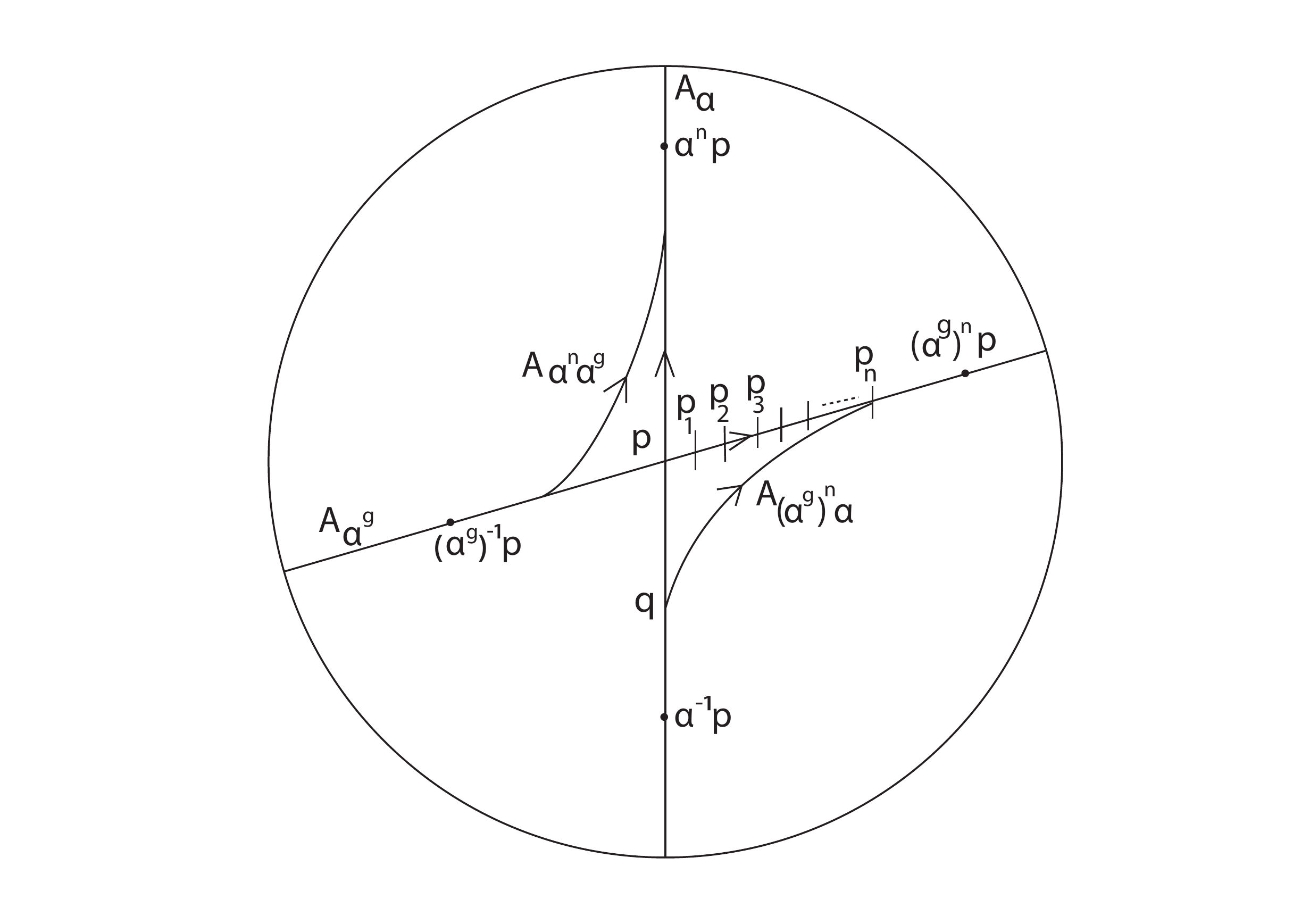}
     \caption{}\label{filling}
\end{figure} 

Consider the Figure \ref{filling}. By the geometry of the products, for each $i$, there exists a lift $z_i$ of the geodesic $z$ intersecting the geodesic segment from $p_{2i}$ to $p_{2i+2}$ and these geodesics are pairwise distinct and disjoint. Also number of geodesic lifts of $z$ that intersect the geodesic segment from $p$ to $q$ is at most $i(z,\A)$. 

Choose $n>N>2i(z,\A)$ such that $N$ is even. If possible, suppose that no lift of $z$ intersects the axis of $(\A^g)^n\A$. Then the convexity of geodesic triangles and uniqueness of geodesic segments implies each $z_i$ intersects the geodesic segment from  $p$  to $q$ and there are atleast $\frac{N}{2}$  of them. As $N>2i(z,\A)$, number of distinct lifts intersecting the geodesic segment from $p$ to $q$ is at least $i(z,\A)+1$ a contradiction. 
\end{proof}

The proof of Lemma \ref{proxymainlemma} follows at once from Lemma \ref{liftequivalence}, Lemma \ref{le1}, Lemma \ref{le2} and Lemma \ref{filling}.
%\section{Computation in Mathematica}
%We checked our results using \emph{Mathematica} for three examples.
%\begin{itemize}
%\item $F$ is pairs of pants and $\A$ is the figure eight. In this case there is only one self-intersection. The computation was performed over $30$ pairs of curves corresponding to this intersection. 
%\end{itemize}

%Our first example is the figure eight in pairs of pants $P$. We used the construction of the Fuchsian representations of $\pi_1(P)$ into $PSL_2(\R)$ from \cite{Mas}. We used the traces of these matrices We computed the lengths of the pairs of curves upto order $30$  

\begin{proof}[Proof of Main theorem]
Consider two intersection points $P_1$ and $P_2$ between $\A$ and $\B$ such that $\la\A*_{P_1}\B\ra=\la\A*_{P_2}\B\ra$. Therefore by Section 3, there exist $g,h\in \pi_1(F)$ such that $\la\A \B^g\ra=\la\A*_{P_1}\B\ra$, $\la\A \B^h\ra=\la\A*_{P_2}\B\ra$   and  they are equal to each other. Using the same proof as in Lemma \ref{le1} we have, for any hyperbolic metric $X$ on $F$, $l_X(\A^n\B^g)=l_X(\A^n\B^h)$ for all $n\in \mathbb{N}$. Now we use the same proof as the in Lemma \ref{le2}. In {\bf{\underline{Case-I}}} if $\A^n\B^g$ is conjugate to $\A^n\B^h$ then from the proof we have $P_1=P_2$ 
(because after a translation their vertices coincides) which gives a contradiction. The proof of {\bf{\underline{Case-II}}} is exactly the same.
\end{proof}

\end{document}